\documentclass[12pt]{amsart}
\usepackage{amsfonts,latexsym,rawfonts,amsmath,amssymb,amsthm,graphicx}

\numberwithin{equation}{section}
\newtheorem{theorem}{Theorem}[section]
\newtheorem{lem}[theorem]{Lemma}
\newtheorem{thm}[theorem]{Theorem}
\newtheorem{pro}[theorem]{Proposition}
\newtheorem{cor}[theorem]{Corollary}

\newcounter{Cnumber}

\def\s{\,\,\,\,}

\def\dint{\displaystyle{\int}}
\def\mv{1.7ex}
\def\endproof{$\hfill\Box$\\}
\def\R{\mathbb{R}}

\def\C{\mathbb{C}}
\def\CP{\mathbb{CP}}

\def\g{\overline{g}}
\def\Sdk{\Sigma_{\delta,\epsilon_k}}
\def\p{\partial}
\def\va{\varphi}
\def\c{{\mathcal C}}

\title[extendability of Conformal structures on punctured surfaces]
{extendability of Conformal structures on punctured surfaces}
\author[J. Chen \& Y. Li]
{Jingyi Chen and Yuxiang Li}
\address{ Department of Mathematics\\ The University of British Columbia, Vancouver, BC V6T1Z2, Canada}
\email{jychen@math.ubc.ca}
\address{Department of Mathematical Sciences, Tsinghua University, Beijing 100084, China}
\email{yxli@math.tsinghua.edu.cn}
\thanks{The first author acknowledges the partial support from NSERC. The second author is partially supported by NSFC. }
\begin{document}
\maketitle

\begin{abstract}
For a smooth immersion $f$ from the punctured disk $D\backslash\{0\}$ into $\R^n$ extendable continuously at the puncture, if its
mean curvature is square integrable and the measure of $f(D)\cap B_{r_k}=o(r_k)$ for a sequence $r_k\to 0$, we show that the Riemannian surface $(D_r\backslash\{0\},g)$ where $g$ is the induced metric is conformally equivalent to the unit Euclidean punctured disk, for any $r\in(0,1)$. For a locally $W^{2,2}$ Lipschitz immersion $f$ from the punctured disk $D_2\backslash\{0\}$ into $\R^n$, if $\|\nabla f\|_{L^\infty}$ is finite and the second fundamental form of $f$ is in $L^2$, we show that there exists a homeomorphism $\phi:D\to D$ such that
$f\circ\phi$ is a branched $W^{2,2}$-conformal immersion from the Euclidean unit disk $D$ into $\R^n$.
\end{abstract}

\section{Introduction}

In two dimensional variation, especially conformally invariant, problems in differential geometry, it is important to know whether islolated singularities are removable while preserving conformal properties. 
In this paper, we shall study the problem of extending conformal structure of an immersion from a punctured 2-dimensional disk across the puncture as a branched conformal immersion.  For an immersion of class $C^{2,\alpha}$ from a punctured disk with bounded mean curvature $H$ and has a unique limit point at the puncture, Gulliver showed in \cite{Gulliver} that there exists a $C^{1,\alpha}$ conformal mapping from the disk such that its restriction to the punctured disk is a $C^{2,\alpha}$ parametrization of the immersion. The minimal surface case was proved by Osserman in \cite{Osserman}. We shall relax the pointwise bound on $H$ by integral bounds. We shall also be concerned with immersions from a punctured disk with lower regularity. For immersions from a disk (not punctured) of class $C^{1,\alpha}$ (the first fundamental form is $C^\alpha$), the classical theorem of Korn-Lichtenstein asserts existence of isothermal coordinates (a simplified proof was given by Chern \cite{Chern}). When the induced metric is merely bounded measurable with $g_{11}g_{22}-g_{12}^2\geq c$ almost everywhere for some positive constant $c$, Morrey's measurable Riemann mapping theorem states that there is a homeomorphism from the unit disk to a neighborhood of a point in the immersed disk satisfying the conformality conditions almost everywhere. Yet, in this generality, the reparametrized immersion does not admit high regularity.
It is advantageous to work with immersions that are in the space $W^{2,2}\cap W^{1,\infty}$ since the $L^2$-norm of the second fundamental form can then be defined.

In Section 2, we show in Theorem \ref{t1} that for a smooth immersion $f$ from the punctured disk $D\backslash \{0\}$ into $\R^n$ which admits a continuous extension at $0$ sending $0\in D$ to $0\in\R^n$,  if its mean curvature $H$ is square integrable and the length of $f(D)\cap \p B_{r_k}\to 0$, or more generally, $\mu(f(D)\cap B_{r_k})=o(r_k)$, along some sequence $r_k\to 0$, then $(D_r\backslash\{0\}, f^*g_{\R^n})$ is conformal to $(D\backslash\{0\},g_{\R^2})$ for any $0<r<1$.  The reason of taking a smaller disk is due to lack of control on $f$ near $\p D$. This generalizes the result in \cite{Gulliver} for $H$ is assumed to be pointwise bounded by a uniform constant therein.

In Section 3, we consider $W^{2,2}\cap W^{1,\infty}$ immersions $f$ from the punctured disk that admits a continuous extension at the puncture. There are two problems to be resolved. First, we need to find a conformal structure $\c$ compatible with the induced metric $g$ by $f$ away from the singularity. This does not directly follow from Morrey's result if we require the reparametrized mapping $f\circ\varphi$ lies in $W^{1,\infty}$: there exists a $W^{2,2}$ map $\varphi:D\rightarrow D$, such that $\varphi^*(g)=e^{2u}g_0$, however, it is not easy to deduce $u\in L^\infty$ directly from Morrey's arguments. Second, we need to determine the conformal type of the punctured disk with respect to the conformal structure $\c$ (a Riemann surface). In fact, we shall show that the Riemann surface is conformally equivalent to the punctured Euclidean disk. Hence, the composition of $f$ with a reparametrization
extends across the puncture as a branched $W^{2,2}$-conformal immersion.

Our strategy is as follows. In section 3.1, we establish a key technical result (Proposition \ref{t2}): for a conformal immersion locally in $W^{2,2}$ from an annulus, there is a positive lower bound on the length measured in the induced metric $g$ for all homotopically non-trivial loops. This is used to rule out the conformal types of annuli in Theorem \ref{conformal.W22}. Since there is a conformal structure compatible with the induced metric by a {\it smooth} immersion,  we can apply Proposition \ref{t2} to prove Theorem \ref{t3}:
For a smooth immersion $f:D\backslash\{0\}\to\R^n$ with $A\in L^2$ and $\|\nabla f\|_{L^\infty(D\backslash\{0\})}<\infty$, the Riemann surface $(D_\delta\backslash\{0\},g)$ is conformally diffeomorphic  to $(D\backslash\{0\},g_0)$ for any $0<\delta<1$.

In section 3.2,  we show that there exists a complex structure ${\mathcal C}$ on the punctured disk arising from the Lipschitz map $f$. The coordinate maps defining the complex structure ${\mathcal C}$ are in $W^{2,2}_{loc}\cap W^{1,\infty}_{loc}$, but the transition functions between coordinate charts are holomorphic. This is done by an approximation argument together with results on mappings with square integrable generalized second fundamental forms \cite{KSc,Helein,K-L,M-S}. For existence of isothermal coordinates of Lipschitz immersions with $L^2$-bounded second fundamental form, an approach via the method of moving frames was given in \cite{Ri}. Note that as a positive lower bound (almost everywhere) on the metric tensor is imposed, ${\mathcal C}$ needs not to be extendable across the puncture.

In section 3.3, we derive Theorem \ref{conformal.W22} that can be stated as: For a $W^{2,2}_{loc}\cap W^{1,\infty}$ mapping $f$ from the punctured Euclidean disk (of radius 2) $D_2\backslash\{0\}$ into $\R^n$,  if $f$ is a $C^0$ immersion with $df\otimes df>C(K)\,g_0>0$ almost everywhere on any compact set $K$ in $D\backslash\{0\}$ and its (generalized) second fundamental form is in $L^2$, then there exists a homeomorphism $\phi:D\to D$ such that $f\circ\phi^{-1}$ is a branched $W^{2,2}$-conformal immersion of the Euclidean disk $D$ into $\R^n$ with $0$ as its only possible branch point.  Moreover, the maps $\phi,\phi^{-1}$ are in $W^{2,2}_{loc}\cap W^{1,\infty}_{loc}$ on $D\backslash\{0\}$ with $\phi(0)=0$ and $(f\circ\phi)^*g_{\R^n}=e^{2u}g_{\R^2}$ on $D\backslash\{0\}$, where $u\in W^{1,2}_{loc}\cap L^\infty_{loc}$ on $D\backslash\{0\}$.

Theorem \ref{conformal.W22} generalizes to the global case:
Let $(\Sigma,h)$ be an oriented surface (not necessarily compact) and let
$S\subset\Sigma$ be a finite set. Suppose
$f\in W^{2,2}_{loc}(\Sigma\backslash S,\R^n)$
and $f$ is a $C^0$ immersion on $\Sigma\backslash S$. Assume  that
(1) $g=df\otimes df>C(K)h$ almost everywhere on any compact set $K\subseteq\Sigma\backslash S$, (2) $\|\nabla f\|_{L^\infty(\Sigma,h)}<\infty$ and
(3) $A\in L^2(\Sigma)$.
Then there is a complex structure ${\mathcal C}$ on $\Sigma$ such that $f:(\Sigma,{\mathcal C})\to\R^n$ is a branched $W^{2,2}$-conformal immersion with its branch locus contained in $S$.

The results in this paper remain valid if the ambient space $\R^n$ is replaced by a smooth compact Riemannian manifold $M^n$, via Nash's isometric embedding theorem.

\section{Smooth immersions from a punctured disk}

\begin{thm}\label{t1}
Let $D$ be the open unit disk in $\R^2$. Let $f: {D}\rightarrow\R^n$ be a continuous map with $f(0)=0$ and $f|_{D\backslash\{0\}}$ is a smooth immersion.  Set
$\mu(\Sigma\cap K) = \mu_g(f^{-1}(K))$ for any $K\subset \mathbb R^n$, where $g$ is the induced metric by $f$ on $\Sigma\backslash\{0\}$ where $\Sigma=f(D)$.
 Assume
\begin{enumerate}
\item  $\int_{\Sigma\backslash \{0\}}|H|^2_{g}d\mu_{g}<\infty$, where $H$ is the mean curvature of $\Sigma\backslash\{0\}$;
\item
There exist positive numbers  $\epsilon_k\searrow 0$
such that
$\frac{\mu(\Sigma\cap B_{\epsilon_k}(0))}{\epsilon_k}\rightarrow 0$ as $k\to\infty$.
\end{enumerate}
Then  $(D_r\backslash\{0\},g)$ is conformal to $(D\backslash\{0\},g_0)$ where $g_0$ is the Euclidean metric on $D$ and $D_r$ is the Euclidean open disk of radius $r$ for any $r\in (0,1)$.
\end{thm}

The whole image set $\Sigma$ may not have the structure of a surface at the puncture. We will need to know that it admits the generalized mean curvature that is in $L^2$.

\begin{lem}\label{1} Under the assumptions in Theorem \ref{t1},
$\Sigma_r=f(D_r)$ is a rectifiable integral 2-varifold with generalized mean curvature in $L^2$, for any $0<r<1$.
\end{lem}

\proof We first show that $\Sigma_r$ has finite total measure. By assumption (2), there is an integer $k_0$ such that $\mu(\Sigma_r\cap B_{\epsilon_{k_0}}(0))\leq\mu(\Sigma\cap B_{\epsilon_{k_0}}(0)) < 1$. The set $\Sigma_r\backslash B_{\epsilon_{k_0}}(0)$ is compact and $g$ is bounded there, so it has finite measure. Hence $\mu(\Sigma_r)<\infty$.

For any $y\in\Sigma_r$, let $\theta(y)={\mathcal H}^0(f^{-1}(y))$ where ${\mathcal H}^0$ is the 0-dimensional Hausdorff measure. By the general area formula, see 8.4 in \cite{Simon},
$$
\int_{\Sigma_r\backslash B_{\epsilon_k}(0)} \theta(y)d\mathcal H^2(y)=\mu(\Sigma_r\backslash B_{\epsilon_k}(0))\leq \mu(\Sigma_r)<\infty.
$$
Letting $k\to\infty$, we see $\theta(y)$ is integrable on $\Sigma_r\backslash \{0\}$.  Further, being the image of a smooth immersion, $\Sigma_r\backslash\{0\}$ is a countable union of embedded surfaces, hence, it is countably 2-rectifiable (Lemma 11.1, \cite{Simon}). It follows that $(\Sigma,\theta)$ is a rectifiable integral 2-varifold (p.77, \cite{Simon}).

Let $\eta$ be a cut-off function with values between 0 and 1  with $|\eta'|\leq C$, and it equals 1 on $[1,+\infty)$ and 0 on $(-\infty,\frac{1}{2})$. Then
$$
\eta_\epsilon(x)=\eta\left(\frac{|f(x)|}{2\epsilon}\right),\,\,\,\,\,x\in D_r
$$
 is 0 when $f(x)\in \Sigma_r\cap B_\epsilon(0)$ and equals 1 when
$f(x)\in\Sigma_r\backslash B_{2\epsilon}(0)$; $\eta_\epsilon$ is continuous on $D_r$ since $f$ is continuous and $\eta_\epsilon$ is locally Lipschitz on $D_r\backslash\{0\}$. Moreover, as $g$ is the induced metric on $\Sigma_r\backslash\{0\}$ by the immersion $f|_{D_r\backslash\{0\}}$, we have, by Kato's inequality, on $D_r\backslash\{0\}$ the estimate
\begin{equation}\label{cutoff}
|\nabla_g\eta_\epsilon|\leq \frac{C}{2\epsilon}\left|\nabla_g|f|\right|
\leq \frac{C}{2\epsilon}|\nabla_gf|=\frac{C}{2\epsilon}\sqrt{2}.
\end{equation}
For any $C^1$ vector field $X$ on ${\mathbb R}^n$, $\eta_{\epsilon_k}X$ is a $C^1$ vector field along $\Sigma_r$ as $\eta_{\epsilon_k}$ vanishes on $\Sigma_r\cap B_{\epsilon_k}(0)$. Then
\begin{equation}\label{X}
-\int_{\Sigma_r} H \cdot \eta_{\epsilon_k}X=\int_{\Sigma_r} \mbox{div}_{\Sigma_r}(\eta_{\epsilon_k}X)=
\int_{\Sigma_r} X\cdot \nabla_g\eta_{\epsilon_k}+\int_{\Sigma_r}\eta_{\epsilon_k} \mbox{div}_{\Sigma_r} X.
\end{equation}
Since $\Sigma_r$ is bounded, $X|_{\Sigma_r}$ is bounded in $C^1$. Then
$$
\left|\int_{\Sigma_r} H\cdot \eta_{\epsilon_k}X\right|\leq C \left(\int_{\Sigma_r \backslash\{0\}} |H|_g^2d\mu_g\right)^{1/2}\mu(\Sigma_r\cap B_{2\epsilon_k}(0))^{1/2}\rightarrow 0\,\,\,\,\mbox{as} \,\,k\to\infty
$$
by assumptions (1) and (2), and
$$
\left|\int_{\Sigma_r} X\cdot \nabla_g\eta_{\epsilon_k}\right| \leq \frac{C}{\epsilon_k}\mu(\Sigma_r\cap (B_{2\epsilon_k}(0))\rightarrow 0\,\,\,\,\mbox{as} \,\,k\to\infty
$$
by \eqref{cutoff} and assumption (2). Letting $k\rightarrow 0$ in \eqref{X}, we have
$$
-\int_{\Sigma_r} H \cdot X=\int_{\Sigma_r} \mbox{div}_{\Sigma_r} X.
$$
As $X$ is arbitrary, $H$ is the generalized mean curvature of $\Sigma_r$ (cf. \cite{Simon}) and by (1) it is in $L^2(\Sigma_r)$. \endproof

The following rigidity result for the Dirichlet problem of harmonic functions
on a punctured disk will be used in the proof of Theorem \ref{t1} to eliminate the conformal types of annuli.

\begin{lem}\label{l2}
Let $\phi$ be a harmonic function on $(D_r\backslash\{0\},g)$ where $g$ is as in Theorem \ref{t1}.
If
$$
\phi|_{\partial D_r}=0, \s
\|\phi\|_{L^\infty}<\infty,\mbox{ and } \int_{D_r\backslash\{0\}}|\nabla_g\phi|^2d\mu_g<\infty,
$$
then $\phi=0$.
\end{lem}

\proof Since $\eta_\epsilon$ is 0 in $B_{\epsilon}(0)$,
$\eta_\epsilon(f)$ is 0 on $f^{-1}(B_\epsilon(0))\cap D_r$. Then
\begin{eqnarray*}
\left|\dint_{ f^{-1}(B_{2\epsilon}(0))\cap D_r}\nabla_g\phi \nabla_g\left(\eta_\epsilon(f) \phi \right)d\mu_g\right|
&=&
\left|\dint_{f^{-1}(B_{2\epsilon}\backslash B_\epsilon(0))\cap D_r}\nabla_g\phi\nabla_g\left(\eta_\epsilon(f) \phi \right)d\mu_g\right|\\
&=&\left|\dint_{
f^{-1}(B_{2\epsilon}\backslash B_\epsilon(0))\cap D_r}\left(\phi \nabla_g\phi\nabla_g\eta_\epsilon(f)
+\eta_\epsilon(f)|\nabla_g\phi|^2\right)d\mu_g\right|\\[\mv]
&\leq&\|\phi\|_{L^\infty}\left(\dint_{
f^{-1}(B_{2\epsilon}\backslash B_\epsilon(0))\cap D_r}|\nabla_g\phi|^2d\mu_g\right)^{\frac{1}{2}}\left(\frac{\mu(\Sigma_r\cap B_{2\epsilon}(0))}{\epsilon^2}\right)^{\frac{1}{2}}\\[\mv]
&&+\dint_{
f^{-1}(B_{2\epsilon}\backslash B_\epsilon(0))\cap D_r}|\nabla_g\phi|^2d\mu_g,
\end{eqnarray*}
where we used the following identity
$$\mu_g(f^{-1}(B_{2\epsilon}\backslash B_\epsilon(0))\cap D_r)=
\mu(\Sigma_r\cap (B_{2\epsilon}\backslash B_\epsilon(0))).$$
By Lemma \ref{1}, we can apply Simon's monotonicity formula for surfaces with square integrable mean curvature \cite{Simon2}, see also \cite{KSc}, to assert
$$\frac{\mu(\Sigma_r\cap B_{2\epsilon}(0))}{\epsilon^2}<C.$$
Since the Dirichlet energy of $\phi$ is finite over $D_r\backslash\{0\}$ and $\mu_g(f^{-1}(B_{2\epsilon}\backslash B_\epsilon(0))\cap D_r)\rightarrow 0$
as $\epsilon\to 0$, it follows from the
the continuity of integration that
$$\lim_{\epsilon\rightarrow 0}
\int_{f^{-1}(B_{2\epsilon}\backslash B_\epsilon(0))\cap D_r}|\nabla_g\phi|^2d\mu_g=0,$$
which in turn implies
$$\lim_{\epsilon\rightarrow 0}\left|\int_{f^{-1}(B_{2\epsilon}(0))\cap D_r}\nabla_g\phi\nabla_g\left(
\eta_\epsilon(f) \phi \right)d\mu_g\right|=0.$$
Since $\eta_\epsilon(f)\phi$ is smooth on $D_r$ and
is 0 in a neighborhood of $0$, by the harmonicity of $\phi$,
$$
\int_{D_r} \nabla_g\phi\nabla_g\left(\eta_\epsilon(f)\phi\right)=0.
$$
Then we get
\begin{eqnarray*}
\int_{D_r}|\nabla_g\phi|^2d\mu_g&=&\lim_{\epsilon\to 0}\int_{D_r\backslash f^{-1}(B_{2\epsilon}(0))}|\nabla_g\phi|^2d\mu_g\\
&=&\lim_{\epsilon\to 0}\int_{D_r\backslash f^{-1}(B_{2\epsilon}(0))}\nabla_g\phi\nabla_g\left(\eta_\epsilon(f)\phi \right) d\mu_g \\
&=&\lim_{\epsilon\to 0} \int_{D_r}\nabla_g\phi\nabla_g\left(\eta_\epsilon(f)\phi \right) d\mu_g
-\lim_{\epsilon\to 0}\int_{ f^{-1}(B_{2\epsilon}(0))\cap D_r}\nabla_g\phi\nabla_g\left(\eta_\epsilon(f)\phi \right) d\mu_g\\
&=&0
\end{eqnarray*}
Therefore $\phi$ must be a constant and identically 0 as it vanishes on $\p D_r$. \endproof

\noindent {\it Proof of Theorem \ref{t1}}:  Since $g$ is positive definite on $D\backslash\{0\}$, by the uniformization theorem for Riemann surfaces, the punctured Riemannian disk $(D\backslash\{0\},g)$ is conformally equivalent to one and only one of the three: (i) a finite annulus $(D\backslash \overline D_{r_0},g_0)$ for some $r_0\in(0,1)$, (ii) the punctured plane $({\mathbb C}\backslash\{0\},g_0)$, (iii) the punctured disk $(D\backslash\{0\},g_0)$. Here $g_0$ denotes the Euclidean metric. Now we fix a positive number $r<1$.

In Case (i), $(D_r\backslash\{0\},g)$ is conformal to $(D\backslash\overline{D}_{r_0})$ for some $r_0\in(0,1)$. 
There exists a nonconstant bounded harmonic function $h$ on $(D\backslash\overline{D}_{r_0},g_0)$ with finite energy, which vanishes on either the inner circle $\partial D_{r_0}$ or the outer circle $\partial D$ but not both. Note that both harmonicity and the energy of $h$ are invariant under conformal diffeomorphisms of the 2-dimensional domain. Moreover, any conformal diffeomorphism $\psi$ between $(D_r\backslash\{0\},g)$ and $(D\backslash \overline{D}_{r_0},g_0)$ either maps $\partial D_r\to \partial D$ or $\partial D_r\to \partial D_{r_0}$. Without loss of any generality we may assume that $h\circ\psi$ equals 0 on $\partial D_r$. However, Lemma \ref{l2} asserts $h\circ\psi$ must be 0. This shows that Case (i) cannot happen.


In Case (ii),  there exists a conformal diffeomorphism $\varphi: (D\backslash\{0\},g)\to(\mathbb C \backslash\{0\},g_0)$.
 Then $\varphi(D_r\backslash\{0\})$ stays inside or outside the embedded closed curve $\varphi(\partial D_r)$ in $\C$. In the former case, $(D_r\backslash\{0\},g)$ is conformally diffeomorphic to $(D\backslash\{0\},g_0)$ by the Riemann mapping theorem since $\partial D_r$ bounds a simply connected domain in $\C$ by the Jordan curve theorem. In the latter case, by using an inversion $\frac{z-p}{|z-p|^2}$ for some point $p$ in the interior of the bounded domain enclosed by $\partial D_r$, we see that $(D_r\backslash\{0\},g)$ is conformally equivalent to a bounded simply connected domain punctured once, which is conformal to $(D\backslash\{0\},g_0)$.


In Case (iii), $(D_r\backslash\{0\},g)$ is conformally equivalent to a simply connected domain in $D$ with one puncture, hence conformal to $(D\backslash\{0\},g_0)$.
\endproof

The decay rate (2) in Theorem \ref{t1} on the area of $\Sigma$ inside small exterior balls  can be achieved under an assumption on length and this leads to

\begin{cor}
Let $f: D\rightarrow\R^n$ be a continuous map with $f(0)=0$ and $f|_{D\backslash\{0\}}$ is a smooth immersion.  Assume
$\int_{\Sigma\backslash \{0\}}|H|^2_{g}d\mu_{g}<\infty$.  If
$\mathcal{H}^{1}(\Sigma\cap \partial B_{\epsilon}(0))\rightarrow 0$ as $\epsilon\rightarrow
0$, where $\mathcal{H}^1$ is the 1-dimensional Hausdorff measure,
then the conclusion of Theorem \ref{t1} holds.
\end{cor}
\begin{proof} It suffices to show (2) in Theorem \ref{t1} holds on $\Sigma_r$ for any fixed $r\in(0,1)$. The position vector $f$ in $\R^n$ satisfies 
$$
\mbox{div}_\Sigma\, f =2
$$
on $D\backslash\{0\}$ and
$$
\mbox{div}_\Sigma\, f = \mbox{div}_\Sigma\, f^T - f\cdot H
$$
where $f^T$ denotes the tangential part of $f$ and $\mbox{div}_\Sigma$ is the divergence along $\Sigma$. It follows
\begin{equation}\label{Gulliver}
\mbox{div}_\Sigma\, f^T =  2+ f\cdot H.
\end{equation}


Let $\rho(x)$ be the distance from $x$ to $0$ in $\R^n$. For $\rho>0$, $\mu(\Sigma\cap \partial B_\rho(0))$ depends on $\rho$ continuously. By Sard's theorem, the regular values of the restriction of $\rho$ to $\Sigma$ are dense, we may assume that $\epsilon_k>0$ and $\delta>0$ are regular values of $\rho|_\Sigma$ and $\epsilon_k\to 0$ as $k\to\infty$; therefore, for any $0<r<1$, the surface 
$$
\Sigma_{\delta,\epsilon_k}:=\Sigma_r \cap (B_{\delta}(0)\backslash B_{\epsilon_k}(0))
$$ 
has compact closure and smooth boundary $\Gamma_\delta \cup \Gamma_{\epsilon_k}$. Integrating \eqref{Gulliver} over $\Sigma_{\delta,\epsilon_k}$ leads to
\begin{eqnarray}\label{volume1}
2\mu(\Sdk)&\leq& 2\delta\int_{\Sdk}|H|d\mu_g
+\int_{\Sdk}\mbox{div}_\Sigma\, f^T d\mu_g \nonumber\\
&\leq&2\delta\mu(\Sdk)^{1/2}\left(\int_{\Sdk} |H|^2\right)^{1/2}
+\int_{\Gamma_\delta\cup\Gamma_{\epsilon_k}}f^T\cdot \nu\nonumber\\
&\leq&\mu(\Sdk)+\delta^2\int_{\Sdk}|H|^2+\int_{\Gamma_\delta\cup\Gamma_{\epsilon_k}}f^T\cdot \nu\nonumber
\end{eqnarray}
where $\nu$ is the unit outward normal to $\Sdk$ at its boundary $\Gamma_\delta\cup\Gamma_{\epsilon_k}$. Noting that $\nu$ is tangent to $\Sdk$,  $f^T\cdot\nu=f\cdot \nu\leq 0$ on $\Gamma_{\epsilon_k}$.
Letting $k\to\infty$ and using $|f^T\cdot\nu|\leq |f|$, we see
\begin{equation}\label{volume2}
\mu(\Sigma_r\cap B_\delta(0))\leq \delta^2\int_{\Sigma\cap B_\delta(0)\backslash\{0\}}|H|^2+\delta \mu(\Gamma_\delta) 
\end{equation}
where $\mu(\Gamma_\delta)=\mu_g(f^{-1}(\partial B_\delta(0)))$ is the total length of the preimage curve of $\Gamma_\delta$ in ${D}$ measured in $g$.
Replacing $\delta$ by $\epsilon_k$ in \eqref{volume1} and using the assumption that ${\mathcal H}^1(\Sigma_r\cap\partial B_{\epsilon_k}(0))\to 0$ as $k\to\infty$ and $|H|^2$ is integrable over $\Sigma\backslash\{0\}$, it is evident that $\mu(\Sigma_r\cap B_{\epsilon_k}(0))/\epsilon_k\to 0$ as $k\to\infty$.  \end{proof}

\section{$W^{2,2}$ Lipschitz immersions from a punctured disk}

In this section, we will be mainly concerned with immersions, that are not necessarily smooth, from punctured surfaces into $\R^n$.

\subsection{Conformal type of $W^{2,2}$-conformal immersions of $D\backslash\{0\}$}

Let $\omega$ be twice the standard K\"ahler
form of $\CP^n$ and denote $W^{1,2}_0(\C)$ the space of functions
$v \in L^2_{loc}(\C)$ with $\nabla v\in L^2(\C)$.
We set $J(\varphi)=|D\varphi\wedge D\varphi|$ to denote the
Jacobi of $\varphi$.

\begin{thm}[M\"uller-\v{S}ver\'ak \cite{M-S}] \label{MS}
Let $\varphi \in W^{1,2}_0(\C,\CP^n)$ satisfy
$$
\int_{\C} \varphi^{*}\omega =0 \quad \mbox{ and } \quad
\int_{\C} J(\varphi) \leq \gamma < 2\pi.
$$
Then there is a unique function $v \in W^{1,2}_0(\C)$
solving the equation $- \Delta v = \ast \varphi^{\ast} \omega$ in $\C$
with boundary condition $\lim_{z \to \infty} v(z) = 0$. Moreover
$$
\|v\|_{L^\infty(\C)} +\|\nabla v\|_{L^2(\C)}
\leq C(\gamma) \int_{\C} |\nabla \varphi|^2.
$$
\end{thm}

We include the following result in \cite{K-L} for completeness.
 
\begin{cor}\label{L.infinity} Let $\varphi \in W^{1,2}(D,\CP^n)$ satisfy
$$
\int_{D} J(\varphi) \leq \gamma < 2\pi.
$$
Then there is a  continuous function $v \in W^{1,2}_0(\C)\cap L^\infty(\C)$
solving the equation $- \Delta v = \ast \varphi^{\ast} \omega$ in $D$
and satisfying the estimates
$$
\|v\|_{L^\infty(\C)} + \|\nabla v\|_{L^2(\C)}
\leq C(\gamma) \int_D |\nabla \varphi|^2.
$$
\end{cor}

\proof

Define the map $\varphi':\C \to \CP^{n-1}$ by
$$
\varphi'(z) =
\left\{\begin{array}{ll}
\varphi(z) & \mbox{ if }z \in D\\
\varphi(\frac{1}{\overline{z}}) & \mbox{ if }z \in \C \backslash \overline{D}
\end{array}\right.
$$
and taking the trace of $\varphi$ on $\partial D$ for $\varphi'$ there. Then $\varphi'\in W_0^{1,2}(\mathbb{C},\CP^{n-1})$ and
$$\int_{\mathbb{C}}{\varphi'}^* \omega =0,\s \int_{\mathbb{C}} J(\varphi') = 2 \int_D J(\varphi) .
$$
The desired result then follows from Theorem \ref{MS}. \endproof

\begin{lem}\label{key}
Let $0<a<1$ and $\varphi\in W^{1,2}(D\backslash \overline{{D}_a},\CP^n)$. There is a constant $\epsilon_0>0$ such that
if $\|\nabla\varphi\|_{L^2}<\epsilon_0$,
then we can find $v\in L^\infty(D\backslash \overline{D}_a)$ which solves
the equation $- \Delta v = \ast \varphi^{\ast} \omega$ in $D\backslash \overline{D}_a$
and satisfies the estimates
$$\| v\|_{L^\infty(D\backslash \overline{D}_a)}\leq C(a)\|\nabla \varphi\|_{L^2(D\backslash \overline{D}_a)}.$$
\end{lem}

\proof 
Let
$$
\widetilde\varphi(z)=\left\{\begin{array}{lll}
                 \varphi(\frac{1}{\overline{z}})&\mbox{if}\s1<|z|<\frac{1}{a}\\
                 \varphi(z)&\mbox{if}\s a\leq |z| \leq 1\\
                 \varphi(\frac{a^2}{\overline{z}})&\mbox{if}\s a^2<|z|<a.
\end{array}\right.
$$
It follows $\widetilde\varphi\in W^{1,2}(D_{\frac{1}{a}}\backslash \overline{D}_a, \CP^n)$ as
$$
\int_{D_\frac{1}{a}\backslash \overline{D}_{a^2}}|\nabla\widetilde\varphi|^2=3\int_{D\backslash \overline{D}_a} |\nabla \varphi|^2,\s
\int_{D_\frac{1}{a}\backslash \overline{D}_{a^2}}|\widetilde\varphi|^2\leq C(a)\int_{D\backslash \overline{D}_a} |\varphi|^2
$$
by the conformal invariance of the Dirichlet integral.
Cover the annulus $D_{\frac{1}{a}}\backslash \overline{D}_a$ by countably many open disks such that every point is contained in finitely many such disks. Take a partition of unity subordinates to this cover: there exist smooth functions $\rho_i$ on the annulus such that
\begin{enumerate}
\item   $0\leq \rho_i\leq 1$, and $\mbox{supp}\,\rho_i\subset\subset D_\frac{1}{a}\backslash \overline{D}_{a^2}$.
\item  $\sum_i\rho_i(z)=1$, $\forall z\in D_\frac{1}{a}\backslash \overline{D}_{a^2}$.
\item  For any $z\in D_\frac{1}{a}\backslash \overline{D}_{a^2}$, there is a neighborhood $V$ of $z$, such that
there are only finitely many $\rho_i$ with
$\mbox{supp}\,\rho_i\cap V\neq\emptyset$.
\end{enumerate}

By (3), there are only finitely many $\rho_i$ whose support intersects the compact set $\overline{D}\backslash D_a$ and we label them as $\rho_1$, $\cdots$, $\rho_m$, and assume $\mbox{supp}\,\rho_i\subset D_{r_i}(z_i)$, where $D_{r_i}(z_i)\subset D_\frac{1}{a}\backslash \overline{D}_{a^2}$.
By Corollary \ref{L.infinity}, we can find $v_i\in L^\infty(\C)\cap W^{1,2}(\C)$, which  solves the equation
$$-\Delta v_i=\ast \widetilde\varphi^{\ast} \omega,\s \forall z\in D_{r_i}(z_i),$$
such that
$$\|v_i\|_{L^\infty(\C)}+\|\nabla v_i\|_{L^2(\C)}\leq C\|\nabla\widetilde\varphi\|_{L^2
(D_{r_i}(z_i))}.$$
Then on $D\backslash \overline{D}_a$,
$$
-\Delta \sum_{i=1}^m\rho_iv_i=\ast \widetilde\varphi^{\ast} \omega-2\sum_{i=1}^m\nabla\rho_i\nabla v_i-\sum_{i=1}^mv_i\Delta\rho_i.
$$
Let $v'$ be the solution to the Dirichlet problem in $D$:
$$
\left\{\begin{array}{lll}
-\Delta v'&=&2\sum\limits_{i=1}^m\nabla\rho_i\nabla v_i+\sum\limits_{i=1}^m v_i\Delta\rho_i,\\[\mv]
v'|_{\partial D}&=&0.
\end{array}\right.
$$
Since
$$
\|\Delta v'\|_{L^2(D)}<C(\max_i\|\Delta\rho_i\|_{C^0},\max_i\|\nabla\rho_i\|_{C^0})
\sum_i\|v_i\|_{W^{1,2}(D)}^2<C\|\nabla\varphi\|_{L^2(D\backslash \overline{D_{a}})},
$$
the elliptic estimates implies 
$$\|v'\|_{L^\infty(D)}<C\|\nabla\varphi\|_{L^2(D_\frac{1}{a}\backslash \overline{D}_{a^2})}.$$
Let $v=\sum_{i=1}^m\rho_iv_i+v'$. Then
$$-\Delta v=\ast \varphi^{\ast} \omega,\s z\in D\backslash \overline{D}_a$$
and
$$\|v\|_{L^\infty(D)}<C\|\nabla\varphi\|_{L^2(D)}
$$
where $C$ depends on $\|\Delta\rho_i\|_{C^0}$ and $\|\nabla\rho_i\|_{C^0}$ for $i=1,..., m$.
\endproof

For a conformal immersion $f:D\rightarrow\R^n$, let $G \in W^{1,2}(D,\CP^{n-1})$
be the associated Gau{\ss} map. Here we embed the Grassmannian
$G(2,n)$ of oriented $2$-planes into $\CP^{n-1}$ by sending an
orthonormal basis $\{e_{1},e_2\}$ to $[(e_1+i e_2)/\sqrt{2}]$. Then
\begin{equation}\label{Gauss.map}
K_g e^{2u}  = \ast G^{\ast} \omega \quad \mbox{ and } \quad
\int_D |\nabla G|^2 = \frac{1}{2} \int_D |A|^2\,d\mu_g
\end{equation}
where $K_g$ is the Gau{\ss} curvature and $A$ is the second fundamental form of $f(D)$ (cf. \cite{M-S}).

\begin{pro}\label{t2}
Let $a\in (0,1)$ and $f\in W^{2,2}_{conf,loc}(D\backslash\overline{D_a},\R^n)$. Set $g=df\otimes df$ and let $\Gamma$ be the set of  closed embedded curves that are
nontrivial in $\pi_1(D\backslash \overline{D_a})$. Then 
$$
\inf_{\gamma\in \Gamma}L_{g}(\gamma)>0
$$
where $L_g(\gamma)$ denotes the length of $\gamma$ measured in $g$. 
\end{pro}

\proof
Assume there exists $\gamma_k\in \Gamma$, such that
$L_g(\gamma_k)\rightarrow 0$ as $k\to\infty$. We consider $\gamma_k$ as
a smooth map from $[0,1]$ to $D\backslash \overline{D_a}$
with $\gamma_k(0)=\gamma_k(1)$. Let $g=e^{2u}g_0$. Then
\begin{equation}\label{go to 0}
\int_{\gamma_k}e^u=\int^{1}_0 e^{u(\gamma_k)}\sqrt{(x_k')^2+(y_k')^2}\,dt=L_{ g}(\gamma_k)\rightarrow 0\s\mbox{as $k\to \infty$}.
\end{equation}

{\it Claim.} For any $\epsilon>0$, $\gamma_k\subset
(D\backslash \overline{D_{1-\epsilon}})\cup (D_{a+\epsilon}\backslash \overline{D_a})$ for sufficiently large
$k$.

Assume the claim is not true. Without loss of generality, we may assume
$\gamma_k(0)\in \overline{D_{1-\epsilon}}\backslash
D_{a+\epsilon}$. Since $f\in W^{2,2}_{loc}(D\backslash\overline{D_a},\R^n)$, we can find
$C=C(\epsilon,f)$, such that for any $\gamma\subset D_{1-\frac{\epsilon}{2}}\backslash D_{a+\frac{\epsilon}{2}}$,
$$L_g(\gamma)\geq CL(\gamma),$$
where $L$ is the length of $\gamma$ measured in the Euclidean metric $g_0$.
If $\gamma_k\subset D_{1-\frac{\epsilon}{2}}\backslash
D_{a+\frac{\epsilon}{2}}$, then
$$
L_g(\gamma_k)\geq CL(\gamma)\geq C\pi\epsilon.
$$
If $\gamma_k$ is not in $D_{1-\frac{\epsilon}{2}}\backslash
D_{a+\frac{\epsilon}{2}}$, then we can find $t_0$, such that
$\gamma_k(t_0)\in \partial D_{1-\frac{\epsilon}{2}}\cup \partial
D_{a+\frac{\epsilon}{2}}$, and
$$
\gamma([0,t_0))\subset D_{1-\frac{\epsilon}{2}}\backslash
D_{a+\frac{\epsilon}{2}}.
$$
Thus,
$$
L_g(\gamma_k)\geq L_g(\gamma_k|_{[0,t_0]})\geq CL(\gamma_k|_{[0,t_0]})\geq C|\gamma_{k}(t_0)-
\gamma_k(0)|\geq \frac{C\epsilon}{2}.
$$
It contradicts the fact that $L_g(\gamma_k)\rightarrow 0$. Now the Claim is established.

Without loss of generality, we may assume  $\gamma_k\subset D_{a+\epsilon}\backslash \overline{D_a}$ for sufficiently large
$k$, because $f(\frac{a}{z})$ is also in $W^{2,2}_{conf,loc}(D\backslash \overline{D_a},\R^n)$ and the metric induced by $f(\frac{a}{z})$ is uniformly equivalent to $g$.
By Lemma \ref{key},  there exists a smooth bounded function $v$ on $\C$ solving the equation
 $$
 -\Delta v=Ke^{2u}
 $$
on $D\backslash \overline D_{a}$, where $K$ is the Gau{\ss} curvature of $g$ and $\Delta$ is the Euclidean Laplacian. Noting that $u$ is also a solution to the same equation, $u-v$ is a harmonic function on $D\backslash \overline D_a$ with respect to $g_0$. Consider the harmonic function $w$
on $D\backslash\overline D_{a}$ defined by
$$
w=u-v-\lambda\log |z|,\,\,\,\,\mbox{where}\,\,\lambda=\frac{1}{2\pi r_0}\int_{\partial D_{r_0}}\frac{\partial (u-v)}{\partial r},\s \mbox{and}\s r_0\in(a,1).
$$
Then we can find a holomorphic function $F$ on $D\backslash\overline D_{a}$ with real part $w$ (cf. \cite{Conway}, Theorem 15.1.3). Evidently, $e^F$ is holomorphic on $D\backslash \overline{D}_a$, with
$
|e^F|=e^{w}.
$

Since $|z|\in (a,1)$ and $\|v\|_{L^\infty}<+\infty$, it follows from the definition of $w$ that
$$
w(z)\leq u(z)+C(a,\lambda,\|v\|_{L^\infty}).
$$
Then
\begin{equation*}
\left|\frac{1}{2\pi i}\int_{\gamma_k}\frac{e^{F(\zeta)}}
{\zeta-z}\,d\zeta\right|<C\int_{\gamma_k}|e^F|<
C\int_{\gamma_k}e^{u}\rightarrow 0\s\mbox{as}\s k\to \infty.
\end{equation*}
Let $\gamma_{\epsilon_0}=(a+\epsilon_0)(\cos\theta,\sin\theta)$ for some fixed $\epsilon_0>0$.
The Deformation Invariance Theorem for holomorphic functions then asserts
$$
\frac{1}{2\pi i}\int_{\gamma_{\epsilon_0}}\frac{e^{F(\zeta)}}
{\zeta-z}\,d\zeta=\frac{1}{2\pi i}\lim_{k\rightarrow \infty}\int_{\gamma_k}\frac{e^{F(\zeta)}}{\zeta-z}\,d\zeta=0.
$$
Cauchy's integral formula implies, for $z\in D_{R}\backslash \overline D_{\epsilon_0}, R\in(a+\epsilon_0,1)$,
$$
e^{F(z)}=\frac{1}{2\pi i}\int_{\partial D_R}\frac{e^{F(\zeta)}}
{\zeta-z}\,d\zeta-\frac{1}{2\pi i}\int_{\gamma_{\epsilon_0}}\frac{e^{F(\zeta)}}
{\zeta-z}\,d\zeta=\frac{1}{2\pi i}\int_{\partial D_R}
\frac{e^{F(\zeta)}}{\zeta-z}\,d\zeta
$$
where all loops are positively oriented.
However, the Cauchy integral on the right hand side defines a holomorphic function on entire $D_R$, hence, $e^F$ can be extended to a  holomorphic ${\mathcal F}$ on $D_R$. By Cauchy's integral formula, for any $z\in D_{a/2}$,
\begin{eqnarray*}
\left|{\mathcal F}(z) \right|&=& \left|\frac{1}{2\pi i}\int_{\gamma_k}\frac{\mathcal F(\zeta)}{\zeta-z}\,d\zeta \right|\\
&=&\frac{1}{2\pi }\left|\int_{\gamma_k}\frac{e^{F(\zeta)}}{\zeta-z}\, d\zeta\right|\\
&\leq&\frac{1}{a\pi}\int_{\gamma_k}|e^F|
\end{eqnarray*}
where in the second equality we used the fact that ${\mathcal F}$ equals $e^F$ on $\gamma_k$.
But \eqref{go to 0} then implies
${\mathcal F}(z)=0$ by letting $k\to \infty$, which in turn asserts ${\mathcal F}(z)=0$. Since $z$ is arbitrary in $D_{a/2}$, we conclude that
${\mathcal F}$ vanishes identically on $D_{a/2}$, hence on $D_R$. This further implies that $e^F$ must vanish identically. But this is impossible for it could only happen if $w=-\infty$ everywhere in $D\backslash\overline D_a$. \endproof

\begin{thm}\label{t3}
Let $f:D\backslash\{0\}\rightarrow\R^n$ be a smooth immersion that satisfies
\begin{enumerate}
\item $\|\nabla f\|_{L^\infty(D\backslash\{0\})}<+\infty$,
\item $\int_{D}|A|^2<+\infty$, where $A$ is the second fundamental form of $f$ on $D\backslash\{0\}$.
\end{enumerate}
\noindent Set $g=df\otimes df$.
Then for any $\delta\in(0,1)$, $(D_\delta\backslash\{0\},g)$ is conformal to $(D\backslash\{0\},g_0)$.
\end{thm}

\proof In light of the proof of Theorem \ref{t1}, the punctured Riemannian disk $(D\backslash\{0\},g)$ has only three distinct conformal types, and to establish the theorem, it suffices to prove that $(D\backslash\{0\},g)$ is not conformal to $(D\backslash \overline{D}_{a},g_0)$ for any $a\in(0,1)$.

Suppose there is a conformal diffeomorphism
$$
\varphi:D\backslash \overline{D}_a\rightarrow D\backslash\{0\}
$$
 such that
$$
\g(z):=(\varphi^*g)(z)=e^{2u(z)}g_0(z),\s z\in D\backslash \overline{D}_a.
$$
Then $f\circ\varphi$ is a conformal immersion of $D\backslash
\overline{D_a}$ in $\R^n$ and
$\bar{g}=d(f\circ\varphi)\otimes d
(f\circ\varphi)$.

Let $\gamma_\epsilon(\theta)=(\epsilon\cos\theta,\epsilon\sin\theta)$ on
$D\backslash\{0\}$. Since $\|\nabla f\|_{L^\infty(D)}<+\infty$, the length of $\gamma_\epsilon$ in the induced metric $g$ by $f$ satisfies
$$
L_g(\gamma_\epsilon)=\int_0^{2\pi}\sqrt{g(\gamma_\epsilon'(\theta),
\gamma_\epsilon'(\theta))}\,d\theta=O(\epsilon)
\rightarrow 0\s \mbox{ as }\s\epsilon\rightarrow 0.
$$
Therefore, $L_{\bar{g}}(\varphi^{-1}(\gamma_\epsilon))=L_g
(\gamma_\epsilon)\to 0$ as $\epsilon\to 0$, which contradicts Proposition \ref{t2}.
\endproof

\noindent{\bf Remark.} 
Theorem \ref{t3} is not true if we replace $\|\nabla f\|_{L^\infty}<+\infty$ with $\mu(f)<+\infty$. For example,
$f(z)=f(re^{i\theta})=(e^{i\theta},r)$ is a smooth
immersion of $D\backslash\{0\}$ in $\R^3$ with
$\mu(f)+\|A\|_{L^2}<+\infty$. However, 
$f(D\backslash\{0\})=S^1\times (0,1)$ is not conformal to
$D\backslash\{0\}$.


\subsection{Existence of conformal structures of non-smooth immersions away from the puncture} We first define the function space that we will work with.
For $0<\lambda<\Lambda$, we denote
$$
BL^n(D,\lambda,\Lambda)=\{f\in W^{2,2}(D,\R^n):
\lambda g_0< df\otimes df<\Lambda g_0,\s \mbox{a.e.} \, x\in D\}.
$$
Every $f\in BL^n(D,\lambda,\Lambda)$ is a local embedding in the sense of the following lemma.

\begin{lem}\label{local.embedding} Let
$f\in BL^n(D,\lambda,\Lambda)$. Then there exist
$\delta$ and $\lambda_1>0$ which only depend on
 $\lambda$ and $\Lambda$, such that
$$
\frac{|f(x)-f(x')|}{|x-x'|}\geq\lambda_1, \s\forall \, |x|,|x'|<\delta.
$$
\end{lem}

\proof

Assume this is not true. Then we can find
$x_k,x_k'\in D$,
such that $x_k,x_k'\rightarrow 0$ and
$$
\frac{|f(x_k)-f(x_k')|}{|x_k-x_k'|}\rightarrow 0, \s \mbox{as $k\to\infty$}.
$$
Write $x_k'=x_k+r_ke_k$ for some unit vector $e_k$ for $r_k=|x_k-x_k'|$ and define
$$
f_k(x)=\frac{f(x_k+r_kx)-f(x_k)}{r_k}.
$$
Without loss of generality, we assume $e_k\rightarrow e:=(0,1)\in\R^2$.
Then $f_k(e)\rightarrow 0\in\R^2$ as
\begin{eqnarray*}
|f(x_k+r_ke)-f(x_k)|&\leq& |f(x_k+r_ke_k)-f(x_k)|+|f(x_k+r_ke)-f(x_k+r_ke_k)|\\
&\leq& \mbox{o}(r_k)+\Lambda r_k|e_k-e|.
\end{eqnarray*}
It is easy to check that
$$
\int_{D_2}|\nabla^2f_k|^2\rightarrow 0,\s
\|\nabla f_k\|_{L^\infty(D_2)}<\Lambda.
$$
Then we may assume $f_k\rightarrow f_0=ax^1+bx^2$
for some $a$, $b\in\R^n$
 strongly  in $W^{2,2}(D)\cap C^0(D)$. Thus
$f_0(0,1)=0$ which implies $b=0$.

However,  since $df_k\otimes df_k(x)=df\otimes
df(x_k+r_kx)$,  we have $df_k\otimes df_k\geq
\lambda((dx^1)^2+(dx^2)^2)$ for a.e. $x$, and $df_k\otimes df_k$ converges
strongly in $L^p$ for any $p>0$. Thus $b\neq 0$ as $\lambda>0$. A contradiction.
\endproof

Next, we show that every $f$ in $BL^n(D,\lambda,\Lambda)$, which is an embedding in the sense of Lemma \ref{local.embedding}, can be approximated by smooth ones with similar geometric properties.
\begin{lem}\label{approximation}
Let $f\in BL^n(D,\lambda,\Lambda)$ and assume $f$ is an embedding. Then we can find
$f_k\in C^\infty(D,\R^n)$ such that for any fixed $r\in(0,1)$, $f_k$ is an
embedding on $D_r$ and $f_k$
converges to $f$ strongly in $W^{2,2}(D_r)$. Moreover, we have for large $k$
\begin{equation}\label{A}
\frac{\lambda}{2}g_0<df_k\otimes df_k<2\Lambda g_0\,\,\,\mbox{on $D_r$},\s
\lim_{k\rightarrow+\infty}\int_{D_r}|A_{f_k}|^2=
\int_{D_r}|A_f|^2d\mu_f.
\end{equation}
\end{lem}

\proof
Take a nonnegative cut-off function $\eta$ supported in $D$ with
$\int_{\R^2}\eta=1$. Define
\begin{equation}\label{app}
f_k(x)=\int_{\R^2}f(x-t_ky)\eta(y)d\sigma_y,\s \mbox{where}\s x\in D_r\s
\mbox{and} \s t_k\rightarrow 0.
\end{equation}
Then $f_k$ is smooth and converges to $f$ strongly in $W^{2,2}(D_r)$.

First, we prove that
$$
df_k\otimes df_k(x)\geq \frac{\lambda}{2}g_0
$$
when $x\in D_r$ and $k$ is sufficiently large. Assume this is not true. Then we can find $x_k\in D_r$, such that
$$
df_k\otimes df_k(x_k)<\frac{\lambda}{2}g_0
$$
and we may assume $x_k\rightarrow x_0\in \overline{D}_r$. By \eqref{app}, we have
$$
\nabla f_k(x_k)=\int_{\R^2}(\nabla_xf(x_k-t_ky))\eta(y)d\sigma_y=
-\int_{D}\left(\nabla_y\frac{ f(x_k-t_ky)-f(x_k)}{t_k}\right)\eta(y)d\sigma_y.
$$
Let
$$
f_k'(y)=\frac{f(x_k-t_ky)-f(x_k)}{t_k}.
$$
Since
$$
|\nabla_y f_k'|<\sqrt{2\Lambda}
$$
and
$$
\int_{D}|\nabla_y^2 f_k'|^2d\sigma_y=
\int_{D_{t_k}(x_k)}|\nabla^2f|^2d\sigma\rightarrow 0,
$$
we may assume that $f_k'(y)$ converges to
 $ay^1+by^2$ strongly in $W^{2,2}(D)$ for some $a,b\in\R^n$. Then we get
$$
\nabla f_k(x_k)\rightarrow\int_{D}(a,b)\eta=(a,b).
$$
However, since $df_k'\otimes df_k'\geq\lambda g_0$, we have
$
|a|^2,|b|^2\geq\lambda.
$
It contradicts  the choice of $x_k$. Thus we can conclude
$df_k\otimes df_k\geq\frac{\lambda}{2}g_0$. The upper estimate $df_k\otimes df_k\leq 2\Lambda g_0$ is obvious.

Next, we prove that $f_k$ is an embedding on $D_r$ for
 large $k$. Since $f_k$ is smooth with
$df_k\otimes df_k\geq\frac{\lambda}{2}g_0$, we can find $\delta>0$,
such that $f_k$ is an embedding on $D_{\delta}(x)$ for
any $x\in D_r$.
Assume $f_k$ is not an embedding on $D_r$. Then
we can find  $x_k$ and $x_k'$ in $D_r$ with $|x_k-x_k'|\geq\delta>0$ and
$f_k(x_k)=f_k(x_k')$. Up to a subsequence, assume $x_k\rightarrow x_0$ and $x_k'\rightarrow x_0'$. Since $f_k$
converges to $f$ in $C^0(D_r)$, we get $f(x_0)=f(x_0')$ which
contradicts that $f$  is an embedding. Therefore, $f_k$ is an embedding for all large $k$.

Lastly, we prove the equality in \eqref{A}.
Since $f_k$ converges strongly to $f$ in $W^{2,2}(D_r)$ for any $r$,  $\nabla f_k$ converges strongly in $L^2$. Then a
subsequence of $\nabla f_k$ converges almost everywhere. By Egorov's Theorem, for any
$\delta>0$,
we can choose $E$, s.t. $\mu(D_r\backslash E)<\delta$,
and $\nabla f_k$ converges uniformly on $E$. The second fundamental form $A_k$ is the normal component of the Hessian of $f_k$: for $1\leq i,j,p,q,m\leq 2$
\begin{equation}\label{2nd f.f.}
A_{k,ij}=\frac{\partial^2f_k}{\partial x^ix^j}-
\frac{\partial^2 f_k}{\partial x^ix^j}\cdot \sum_{m,p}\frac{\partial f_k}
{\partial x^m}g_k^{mp}\frac{\partial f_k}{\partial x^p},
\end{equation}
where 
$(g_k^{pq})=\left(\frac{\partial f_k}{\partial x^p}\cdot\frac{\partial f_k}{\partial x^q}\right)^{-1}$ is the inverse matrix of $g_k$.  
Since
$$
g_k^{-1}=\frac{1}{\det(g_k)}\left(\begin{array}{cc}
                      g_{k,22}&-g_{k,12}\\
                      -g_{k,12}&g_{k,11}\end{array}\right)
                      $$
                       and
$$
\frac{\lambda^2}{4}\leq \det(g_k)
$$ 
because $\frac{\lambda}{2}g_0\leq df_k\otimes df_k$, the inverse matrix $g^{-1}_k$ is bounded: $|g_k^{pq}|<C$. Then
the uniform convergence of $\nabla f_k$ on $E$ and the strong convergence of $f_k$ in $W^{2,2}(D_r)$ imply
$$
\int_E|A_{f}|^2d\mu_f=\lim_{k\rightarrow+\infty}
\int_E|A_{f_k}|^2d\mu_{f_k}.
$$
Moreover, using $|\nabla f_k|\leq \sqrt{2\Lambda}$ on $D_r\backslash E$, we have
$$
\left(\int_{D_r\backslash E}|A_{f_k}|^2d\mu_{f_k}\right)^\frac{1}{2}<C\|\nabla^2f_k\|_{L^2(D_r
\backslash E)}\rightarrow C\|\nabla^2f\|_{L^2(D_r
\backslash E)},\s\mbox{as $k\to\infty$}.
$$
Then
$$
\limsup_{k\rightarrow+\infty}
\left|\int_{D_r}|A_{f_k}|^2d\mu_{f_k}-\int_{D_r}|A_f|^2d\mu_{f}\right|
\leq C\int_{D_r\backslash E}|\nabla^2f|^2dx+
\int_{D_r\backslash E}|A_{f}|^2.
$$
Recall $|A_f|\in L^2(D)$ by assumption. Letting $\mu(D_r\backslash E)\rightarrow 0$, we finish the proof. \endproof

Next, we prove that a  bi-Lipschitz immersion of
$D$ in $\R^n$ with small $\|A\|_{L^2}$ must be a $W^{2,2}$ conformal immersion. We begin with

\begin{lem}\label{notapoint}
Let $f_k$ be a smooth conformal immersion of $D$ in $\R^n$ with
\begin{equation}\label{d.delta}
\liminf_{r\rightarrow 1}d_{g_{k}}(0,\partial D_r)\geq
\delta >0
\end{equation}
where $d_{g_{k}}$ is the distance function w.r.t. the metric $g_{k}=df_k\otimes df_k$, $\delta$ is a constant and
$$
\int_D|A_{f_k}|^2d\mu_{f_k}<4\pi-\gamma.
$$
Assume
$g_{k} = e^{2u_k} g_0$,
$\mu_{g_k}(D) \leq \Lambda$ and $f_k(0) = 0$.
Then $f_k$ converges  weakly in $W^{2,2}(D_r)$ to a nonconstant $W^{2,2}$-conformal
map, for any $r<1$. Moreover,
we have
$$
\|u_k\|_{L^\infty(D_r)}+\|\nabla u_k\|_{L^2(D_r)}<C(\gamma,\Lambda,r),\s\forall r\in (0,1). $$
\end{lem}

\proof  By Theorem 5.1.1 in \cite{Helein}, it suffices to show $f_k$ does not converge to a point.
Assume, on the contrary, that $f_k$ converges to a point. Then $u_k$ converges to $-\infty$ uniformly   on $D_\frac{1}{2}$ (cf. \cite{Helein}, \cite{K-L}).
By \eqref{d.delta}, for any fixed $\theta$, we have
$$
\int_0^1e^{u_k(te^{i\theta})}dt\geq \liminf_{r\to 1}  d_{g_k}(0,\partial D_r)\geq\delta.
$$
Then we can find $r_k\in(\frac{1}{2},1)$ such that
 $u_k(r_ke^{i\theta})>c(\delta)$ for all large $k$.

By \eqref{Gauss.map} and Corollary \ref{L.infinity}, there is a function $v_k$ with $\|v_k\|_{L^\infty(D)}<C(\gamma)$ solving
$$
-\Delta v_k=K_{f_k}e^{2u_k}.
$$
Then we have
\begin{enumerate}
\item $(u_k-v_k)\rightarrow-\infty$ uniformly on $D_\frac{1}{2}$;
\item for any fixed $\theta$,
there exists  $r_k\in (\frac{1}{2},1)$, such that $(u_k-v_k)(r_ke^{i\theta})\geq
c'(\delta)$. Here $c'(\delta)$ may be negative.
\end{enumerate}
It follows that $u_k-v_k$ has a minimum in $D$. The strong maximum principle implies the harmonic function $u_k-v_k$ must be constant. However,
$$
c'(\delta)\leq (u_k-v_k)(r_ke^{i\theta})=\min _D(u_k-v_k) \to -\infty\s\mbox{as} \s k\to\infty
$$
yields a contradiction.
\endproof

Now we are ready to show the existence of conformal structures for the weak immersions with small total curvature.

\begin{pro}\label{chart} Let $f\in BL^n(D_2,\lambda,\Lambda)$ and assume $f$ is an embedding with
$$
\int_{D_2}|A_f|^2<4\pi-\gamma.
$$
Then there exists a  bijective map $\varphi:D\rightarrow D$,
such that $\varphi,\varphi^{-1}\in W^{2,2}_{loc}(D)\cap W^{1,\infty}_{loc}(D)$ and $f\circ\varphi\in W^{2,2}_{conf,loc}(D,\R^n)$.
\end{pro}

\proof

By Lemma \ref{approximation}, there exists a sequence of smooth embeddings $\{f_k\}$ converging to $f$ strongly in $W^{2,2}(D_r)$ for any $r\in(0,2)$ with $\|A_{f_k}\|_{L^2(D_r)}\rightarrow\|A_f\|_{L^2(D_r)}$  and $\mu(D_r,g_k)\to \mu(D_r,g_f)<\infty$ as $k\to\infty$, where $g_k,g_f$ are the induced metrics by $f_k,f$, respectively.

Since $f_k$ is smooth,  the Riemannian disk $(D_2,g_k)$ is either conformal to $\C$ or $(D,g_0)$ by the uniformization theorem. So the simply connected proper subdomain 
$(D,g_k)$ is conformal to $(D,g_0)$. There is a diffeomorphism
$\varphi_k:(D,g_0)\rightarrow (D,g_k)$ with
$$
\bar{g}_k:=\varphi_k^*(g_k)=e^{2u_k}\left((dy^1)^2+(dy^2)^2\right)
$$
for some smooth function $u_k$. Further, we can assume 
$
f_k\circ\varphi_k(0)=0.
$
The areas $\mu(D,\bar g_k)=\mu(D,g_k)$ are uniformly bounded by some $\Lambda$ and  for large $k$
$$
\|A_{f_k\circ\varphi_k}\|_{L^2(D,\bar g_k)} = \|A_{f_k}\|_{L^2(D,g_k)}\leq 4\pi-\frac{\gamma}{2}.
$$

Note that $f_k$ converges to $f$ in $C^0(D_2)$ and $f$ is an embedding. For large $k$
$$
d_{\R^n}(0,f_k(\partial D))  > \frac{1}{2}d_{\R^n}(0,f(\partial D) ) :=\delta>0.
$$
Since $f_k\circ\varphi_k$ is an isometry from $(D,\bar g_k)$ to $(f_k(D),g_k)$,  for every $r\in(0,1)$ it holds
$$
d_{\bar g_k}(0,f_k\circ\varphi_k(\partial D_r))=d_{g_k}(0,f_k(\partial D_r)).
$$
As the extrinsic distance is no larger than the intrinsic distance, we have for large $k$
\begin{eqnarray*}
\liminf_{r\to 1}d_{\bar g_k}(0,f_k\circ \varphi_k(\partial D_r))
&=&\liminf_{r\to 1} d_{g_k}(0,f_k(\partial D_r))\\
&\geq& \liminf_{r\to 1} d_{\R^n}(0,f_k(\partial D_r))\\
&>&\frac{1}{2}d_{\R^n}(0,f(\partial D))\\
&=&\delta.
\end{eqnarray*}
Applying Lemma \ref{notapoint} to the conformal maps $f_k\circ\varphi_k$,
\begin{equation}\label{uk}
\|u_k\|_{L^\infty}(D_r)+\|\nabla u_k\|_{L^2(D_r)}<C(\gamma,r,\Lambda)\s\mbox{for any}\s r\in(0,1).
\end{equation}


Let $G_k(y)=(g_{ij}(\varphi_k(y)))$ and let
$J(\varphi_k)$ be the Jacobi of $\varphi_k$. Then
$$
e^{2u_k}I=J^T(\varphi_k)G_k J(\varphi_k).
$$
Hence
$$
J(\varphi_k)J(\varphi_k)^{T}=e^{2u_k}G_k^{-1}.
$$
It follows
\begin{equation}\label{iso}
\lambda(r) g_0<d\varphi_k\otimes d\varphi_k\leq \Lambda(r)g_0,
\end{equation}
where $0<\lambda(r)<\Lambda(r)$. Noting that $|\varphi_k|<1$ as the image of $\varphi_k$ lies in $D$,
we may assume $\varphi_k$ converges in $C^0(D_r)$ to
a map $\varphi$ for any $r\in(0,1)$.

Next, we bound the Hessian of $\va_k$ uniformly in $k$ for any $r\in(0,1)$: 
\begin{equation}\label{2nd}
\|\nabla^2 \varphi_k\|_{L^2(D_r)}<C(r).
\end{equation}

To see \eqref{2nd}, we first let $h_k(y)$ be the ${2\times 2}$-matrix whose entries are given by
$$
h_{k,ij}(y)=\frac{\partial\varphi_k(y)}{\partial y^i}
\frac{\partial\varphi_k(y)}{\partial y^j}=e^{2u_k(y)}g_{k}^{ij}(\varphi_k(y)).
$$
Then
$$
h_k^{-1}(y)=(h_k^{ij}(y))=e^{-2u_k(y)}G(y).
$$
By \eqref{uk}, \eqref{iso} and the fact that $\frac{\lambda}{2}g_0\leq df_k\otimes df_k\leq 2\Lambda g_0$, we have
\begin{eqnarray*}
\int_{D_r}|\nabla h_k^{-1}|^2&\leq&  C\int_{D_r}e^{-4u_k}\left(|\nabla f_k|^4|\nabla u_k|^2+|\nabla\varphi_k|^2|\nabla^2f_k|^2\right)< C(r).
\end{eqnarray*}
Since $\nabla h_k=-h_k(\nabla h_k^{-1})h_k$ and $h_k$ is bounded, we have
$$
\|\nabla h_k\|_{L^2(D_r)}+\|\nabla h_k^{-1}\|_{L^2(D_r)}<C(r).
$$
We now estimate $\nabla^2 \va_k$ as follows, for simplicity, write $\p_i \va_k, \p_{ij}\va_k$ for the first and second order derivatives respectively. Since $\p_i\va_k(y),i=1,2$ form a basis of $T_{\va_k(y)}\va_k(D)$ and $h_k=\langle \nabla\va_k,\nabla\va_k\rangle$, we have
\begin{eqnarray}\label{11}
\p_{11}\va_k&=&\langle \p_{11}\va_k,\p_i\va_k\rangle h^{ij}\p_j\va_k\\
&=&\langle \p_{11} \va_k, \p_1\va_k\rangle h^{1j}\p_j\va_k
+\langle \p_{11}\va_k,\p_2\va_k\rangle h^{2j}\p_j\va_k \nonumber\\
&=&\frac{1}{2}\left(\p_1|\p_1\va_k|^2\right)h^{1j}\p_1\va_k+\left(\p_1\langle\p_1\va_k,\p_2\va_k\rangle-\langle\p_1\va_k,\p_{12}\va_k\rangle\right)h^{2j}\p_j\va_k \nonumber\\
&=&\frac{1}{2}\left(\p_1 h_{k,11}\right)h^{1j}\p_1\va_k+\left(\p_1h_{k,12}-\frac{1}{2}\p_2h_{k,11}\right)h^{2j}\p_j\va_k \nonumber
\end{eqnarray}
and similar for $\p_{22}\va_k$;
\begin{eqnarray}\label{12}
\p_{12}\va_k&=& \langle \p_{12}\va_k,\p_i\va_k\rangle h^{ij}\p_j\va_k\\
&=&\langle \p_{12}\va_k,\p_1\va_k\rangle h^{1j}\p_j\va_k+\langle \p_{12}\va_k,\p_2\va_k\rangle h^{2j}\p_j\va_k \nonumber \\
&=&\frac{1}{2}(\p_2h_{k,11})h^{1j}\p_j\va_k +\frac{1}{2}(\p_1 h_{k,22})h^{2j}\p_j\va_k\nonumber
\end{eqnarray}
Hence \eqref{2nd} holds.

Now, we may assume $\varphi_k$ converges to $\va$ weakly in $W^{2,2}_{loc}(D)$. Thus $\varphi\in W^{2,2}_{loc}(D)\cap W_{loc}^{1,\infty}(D)$
and satisfies
$$
e^{2u(y)}\delta_{ij}=\frac{\partial\varphi^p}{\partial y^i}(y)
\frac{\partial\varphi^q}{\partial y^j}(y)g_{pq}(\varphi(y)),
$$
$$
J(\varphi)J(\varphi)^{T}=e^{2u}G^{-1},
$$
$$
\lambda'g_0<d\varphi\otimes d\varphi<\Lambda'g_0,
$$
where $0<\lambda'<\Lambda$.

Let $\psi_k(x)=\varphi_k^{-1}(x)$.
We have
\begin{equation}\label{isom}
g_{k,ij}(x)=\frac{\partial\psi_k(x)}{\partial x^i}
\frac{\partial\psi_k(x)}{\partial x^j}e^{2u_k(\psi_k(x))}.
\end{equation}
Hence $\|\nabla \psi_k\|_{L^\infty(D_r)}<C$.
Thus we may assume $\psi_k$ converges in $C^0$ to
a map $\psi$. Since $\varphi_k(\psi_k(x))=x$, we
see $\varphi(\psi(x))=x$. Then $\varphi$ is a bijective.

Using similar arguments, we can prove $\|\nabla^2\psi_k\|_{L^2(D_r)}<C(r)$. In fact, we just need to replace $h_k$ by $h_k^{-1}$ and $\va_k$ by $\psi_k$ in \eqref{11} and \eqref{12}.
Then $\psi\in W^{2,2}_{loc}(D)\cap W^{1,\infty}_{loc}(D)$. \endproof

\subsection{Extending conformal structures at the puncture}

With the preparation in sections 3.1 and 3.2, we now state and prove the main result in section 3.

\begin{thm}\label{conformal.W22}
Let $f\in W^{2,2}_{loc}(D_2\backslash\{0\},\R^n)$ be a $C^0$ immersion. Assume that

\begin{enumerate}

\item[1)]
For any $r\in (0,1)$ there is $\lambda(r)>0$, s.t.
$g:=df\otimes df>\lambda(r)g_0$ almost everywhere in $D_2\backslash D_r$, where $g_0$ is the Euclidean metric.
\item[2)] $\|\nabla f\|_{L^\infty(D_2)}<+\infty$.
\item[3)] $\int_{f(D_2\backslash\{0\})}|A|^2<+\infty$, where $A$ is the second fundamental form of $f$ on $D_2\backslash \{0\}$.
\end{enumerate}
Then  there exists a homeomorphism $\phi: D\rightarrow D$ with $\phi(0)=0$ that satisfies 
$\phi,\phi^{-1}\in W^{2,2}_{loc}(D\backslash\{0\})\cap W^{1,+\infty}_{loc}
(D\backslash\{0\})$
and
$
(\phi^{-1})^*g=e^{2u}g_0 \mbox{ on }D\backslash\{0\},$
where $u\in W^{1,2}_{loc}(D\backslash\{0\})\cap
L_{loc}^\infty(D\backslash\{0\}).
$
Moreover, $f\circ\phi^{-1}$ is branched $W^{2,2}$-conformal immersion of
the Euclidean unit disk $D$ in $\R^n$, with $0$ as its only possible branch point.
\end{thm}

\proof By Lemma \ref{local.embedding},
for any $p\in D_2\backslash \{0\}$, we can find a disk $D_\delta(p)
\subset D\backslash \{0\}$, such that $f$ is an embedding on
$D_\delta(p)$ and we may further assume $\|A\|_{L^2(D_\delta(p),g)}^2<4\pi-\gamma$ by 3). By Proposition \ref{chart}, there exists $\varphi_{D_\delta(p)}:D_\delta(p)\rightarrow
D$,  $\varphi_{D_\delta(p)},\varphi^{-1}_{D_\delta(p)}\in W^{1,\infty}\cap
W^{2,2}$ and
\begin{equation}\label{conformal}
\left(\varphi_{D_\delta(p)}^{-1}\right)^*g=e^{2u_{D_\delta(p)}}g_0.
\end{equation}

Thus, we can find a countable open cover $\mathcal{U}=
\{U_i\}_{i\in{\mathcal I}}$ of $D_2\backslash \{0\}$, where $U_i$ is a disk $D_{\delta_i}(p_i)$ as above, such that
there is a homeomorphism $\varphi_i$ from
$U_i$ to an open set of $\R^2$, such that $\varphi_i,\varphi^{-1}_i\in W^{1,\infty}\cap W^{2,2}$. Then $$\phi_{ij}=\varphi_j\circ\varphi_i^{-1}:\varphi_i(U_i\cap U_j)
\rightarrow\varphi_j(U_i\cap U_j)$$
is in $W^{2,2}\cap W^{1,\infty}$ with
$$
e^{2u_i}g_0=\phi_{ij}^*(e^{2u_j}g_0).
$$
Thus, $\phi_{ij}$ is conformal. Furthermore, we can assume that for each pair of $i,j$ the orientation induced by $\va_{ij}$ on $U_i\cup U_j$ is the same as that of the oriented surface $D\backslash\{0\}$. Then $\phi_{ij}$ is holomorphic. Then ${\mathcal A}=\{(U_i,\varphi_i):i\in{\mathcal I}\}$
is an atlas on $D_2\backslash\{0\}$ which defines a smooth structure and a
complex structure ${\mathcal C}$ on $D_2\backslash \{0\}$. The identity map $\mbox{Id}: (D_2\backslash\{0\},{\mathcal C})\to (D\backslash\{0\}, g_0)$, in local coordinates, is $\mbox{Id}\circ\va^{-1}_i(y)$ which is in $W^{2,2}_{loc}\cap W^{1,\infty}_{loc}$. Then $f$ is a $W^{2,2}$-conformal immersion from $(D_2\backslash\{0\},{\mathcal C})\to (\R^n,\langle\cdot,\cdot\rangle)$, by noting that the left side of \eqref{conformal} is $g$ in the chart $(D_\delta(p), \va_{D_\delta(p)})$. Note that the smooth structure determined by 
${\mathcal A}$ is different from the standard one on $D_2\backslash\{0\}$ unless $\va_i$'s are smooth.

We now show that $(D_2\backslash\{0\},{\mathcal C})$ is not conformal to the annulus $(D\backslash\overline{D}_a,g_0)$ for any $a\in(0,1)$.
If there exists a conformal diffeomorphism $\psi:(D\backslash\overline{D}_a,g_0)\to(D_2\backslash\{0\},{\mathcal C})$, then $\overline{f}=f\circ\psi:(D\backslash\overline{D}_a,g_0)\to\R^n$ is in $W^{2,2}_{conf,loc}$ and $\|\nabla \overline{f}\|_{L^\infty}<\infty$, where we may restrict $\overline{f}$ to a smaller annulus $D_{1-r_0}\backslash\overline{D}_{a+r_0}$ so that $\|\nabla^2\psi\|_{L^\infty}+\|\nabla\psi\|_{L^\infty}$ is bounded.  Then
$$
\lim_{\epsilon\rightarrow 0}L_{\overline g}(\psi^{-1}(\gamma_\epsilon))=\lim_{\epsilon\rightarrow 0}
L_g(\gamma_\epsilon)=0,
$$
where $\gamma_\epsilon=(\epsilon\cos\theta,\epsilon\sin\theta)$. But this contradicts Proposition \ref{t2}.

Then $(D_2\backslash\{0\},{\mathcal C})$ is conformal to $D\backslash\{0\}$ or $\C\backslash\{0\}$. For both cases, $(D\backslash\{0\},{\mathcal C})$ is conformal to $D\backslash\{0\}$. Thus, we can find  a conformal diffeomorphism $\phi:D\backslash\{0\}\rightarrow
(D\backslash\{0\},{\mathcal C})$. Then $\phi,\phi^{-1}\in W^{2,2}_{loc}\cap W^{1,\infty}_{loc}$ over $D\backslash\{0\})$.

Next, we show $\lim_{z\rightarrow 0}\phi(z)=0$. We will use the following:

{\it Fact}: if $\varphi:D\backslash\{0\}\rightarrow
D\backslash\{0\}$ is a continuous homeomorphism, then for
any $z_k\rightarrow 0$, $|\varphi(z_k)|\rightarrow 0$
or $1$, {\it i.e.}, $\varphi(z_k)$ converges to $\partial D\cup
\{0\}$.

This is because $\varphi(z_k)$ cannot have any accumulation points in
$D\backslash\{0\}$: if this were not true, then there would exist a small disk $D(r_0)(q)\subset D\backslash\{0\}$ centered at an accumulation point $q\in D\backslash\{0\}$ and $\varphi^{-1}(D_{r_0}(q))$ contains infinitely many $z_k$; but  this contradicts $x_k\to 0$. Moreover, for a given $\varphi$, only one of the limiting behavior of $|\varphi(z_k)|$ can occur: if there were
$z_k\to 0$ and $z_k'\to 0$ with $|\varphi(z_k)|\to 0$ and $|\varphi(z_k')|\to 1$ then the intermediate value theorem would imply existence of $z_k''\to 0$ with
$|\varphi(z_k'')|=\frac{1}{2}$ hence $z_k''$ would have at least one accumulation point on $\p D_{\frac{1}{2}}$.

By this Fact,  $\lim_{z\rightarrow 0}\phi(z)=0$
is equivalent to $\lim_{w\rightarrow 0}\phi^{-1}(w)=0$.
Now, we prove $\lim_{w\rightarrow 0}\phi^{-1}(w)=0$ by contradiction.
Assume that there exists $w_k\rightarrow 0$, such that
$\phi^{-1}(w_k)\rightarrow p\in\partial D$.
We claim that there exists a small $r>0$ with
$\phi^{-1}(D_r\backslash\{0\})\subset D\backslash \overline{D_\frac{1}{2}}$.
Otherwise, we can find $w_k'\rightarrow 0$ with
$|\phi^{-1}(w_k')|\leq\frac{1}{2}$. By the intermediate value theorem,
there exists $w_k''$ in $\overline{D}_{|w_k|}\backslash
D_{|w_k'|}$ or $\overline{D}_{|w_k'|}\backslash D_{|w_k|}$ such that $|\phi^{-1}(w_k'')|=\frac{1}{2}$. But this is impossible for
$w_k''\rightarrow 0$. The claim asserts
$f\circ\phi\in W^{2,2}_{conf,loc}(D\backslash \overline{D_\frac{1}{2}},\R^n)$, and
$\phi^{-1}(\gamma_\epsilon)\subset D\backslash \overline{D_\frac{1}{2}}$,
where $\gamma_\epsilon=(\epsilon\cos\theta,\epsilon\sin\theta)$.
Then
$$
\lim_{\epsilon\rightarrow 0}L_{ g_{f(\phi)}}(\phi^{-1}(\gamma_\epsilon))=\lim_{\epsilon\rightarrow 0}
L_{g_f}(\gamma_\epsilon)=0.
$$
 But this contradicts Proposition \ref{t2}. \endproof






For general surfaces with punctures, we have:

\begin{cor}\label{global}
Let $(\Sigma,h)$ be an oriented surface (may not compact) and let
$S\subset\Sigma$ be a finite set. Suppose
$f\in W^{2,2}_{loc}(\Sigma\backslash S,\R^n)$
and $f$ is a $C^0$ immersion on $\Sigma\backslash S$. Assume that
\begin{enumerate}
\item[1)]
For any $r\in (0,1)$, there exists $\lambda(r)$, such that it holds a.e. in $ \Sigma\backslash \bigcup_{p\in S}B_r(p)$
$$
g=df\otimes df>\lambda(r)h.
$$
\item[2)]
$\|\nabla f\|_{L^\infty(\Sigma,h)}<\infty$.
\item[3)]
$\int_{\Sigma}|A|^2<\infty$, where $A$ is the second fundamental form of $f$ on $\Sigma\backslash S$.
\end{enumerate}
Then there is a complex structure $c$ on $\Sigma$ such that $f:(\Sigma,c)\to\R^n$ is a branched $W^{2,2}$-conformal immersion with its branch locus contained in $S$.

\end{cor}

\proof
Applying Lemma \ref{local.embedding}  and Proposition \ref{chart}, for any $p\notin S$, we  choose a neighborhood $U$ of $p$ and a homeomorphism $\varphi_U$ from $U$ to $D$, such that $\varphi_U\in W^{2,2}(U,\R^2)\cap W^{1,\infty}(U,\R^2)$ and
\begin{equation}\label{cor}
(\varphi^{-1}_U)^*(g)=e^{2u_U}g_0.
\end{equation}
For $p\in S$, by Theorem \ref{conformal.W22}, there exist a neighborhood $U$ of $p$ with $U\cap S=p$
and a homeomorphism $\varphi_{U}$ from $U$ to $D$ that is in $W^{2,2}_{loc}(U\backslash\{0\},\R^2))\cap W_{loc}^{1,\infty}(U\backslash\{0\},\R^2)$ and $\varphi_{U}(p)=0$ and \eqref{cor} holds. As in the proof of Theorem \ref{conformal.W22},  we have an atlas $\mathcal{A}=
\{(U_i,\va_i):i\in{\mathcal I}\}$ of $\Sigma$, such that
for any $i,j\in{\mathcal I}$ the transition function $\phi_{ij}=\varphi_j\circ\varphi_i^{-1}$ is holomorphic, using the given orientation  on $\Sigma$ to adjust if necessary, and
in $W^{2,2}_{loc}\cap W_{loc}^{1,\infty}$. Thus ${\mathcal A}$
determines a smooth structure and a
complex structure on $\Sigma$. Take a smooth metric $h'$ that is compatible  with the new complex structure $c$. The identity map from $(\Sigma,h')$ to $(\Sigma,h)$ is homeomorphic and is in $W^{2,2}_{loc}(\Sigma\backslash S)\cap W^{1,\infty}_{loc}(\Sigma\backslash S)$. Therefore, $f:(\Sigma,c)\to\R^n$ is a branched $W^{2,2}$-conformal immersion with $S$ as its possible branch locus.
\endproof


\begin{thebibliography}{2}

\bibitem{Chern} S. S. Chern, {\it An elementary proof of the existence of isothermal parameters of a surface}, {Proc. Amer. Math. Soc.} {\bf 6} (1955) 771-782.

\bibitem{Conway} J. B. Conway, Functions of One Complex Variable II, Graduate Texts in Math., Springer, 1996.

\bibitem{Gulliver} R. Gulliver, {\it Removability of singular points on surfaces of bounded
mean curvature}, J. Differential Geom. {\bf 11} (1976) 345-350.


\bibitem{Helein}F. H\'elein, Harmonic maps,
conservation laws and moving frames.
Translated from the 1996 French original.
With a foreword by James Eells. Second edition.
Cambridge Tracts in Mathematics, 150.
Cambridge University Press, Cambridge, 2002.

\bibitem{KSc}E. Kuwert and R. Sch\"atzler, {\it Removability of point singularities of Willmore surfaces},  Ann. of Math. {\bf 160} (2004), 315-357.

\bibitem{K-L} E. Kuwert and Y. Li, {\it $W^{2,2}$-conformal immersions of a closed Riemann surface into $\R^n$},  {Comm. Anal. Geom.} {\bf 20} (2012),  313-340.

\bibitem{M-S} S. M\"uller and V. \v{S}ver\'ak, {\it  On surfaces
of finite total curvature}, {J. Differential Geom.}
{\bf 42} (1995),
229-258.

\bibitem{Osserman} R. Osserman, {\it On Ber's theorem on isolated singularities}, { Indiana Univ. Math. J.} {\bf 21} (1973), 337-342.

\bibitem{Ri} T. Rivi\`ere, {\it Weak immersions of surfaces with $L^2$-bounded second fundamental form}, PCMI Graduate Summer School, 2013.

\bibitem{Simon2} L. Simon, {\it Existence of surfaces minimizing the Willmore functional}, Commu. Anal. Geom. Vol. 1, No. 2, (1993), 281-326.

\bibitem{Simon} L. Simon, Lectures on Geometric Measure Theory, Proceedings of the Centre for Mathematical Analysis, ANU, 1984.


\end{thebibliography}
\end{document}